\documentclass[12pt,reqno]{amsart}

\usepackage[letterpaper,margin=1in,footskip=0.25in]{geometry}
\usepackage{amssymb}
\usepackage[only,llbracket,rrbracket]{stmaryrd}
\usepackage{microtype}
\usepackage{mathtools}
\usepackage{enumitem}
\usepackage[noadjust]{cite}
\usepackage{hyphenat}
\usepackage{color}

\PassOptionsToPackage{dvipsnames}{xcolor}
\usepackage{tikz-cd}

\PassOptionsToPackage{
	pdfusetitle,
	colorlinks,
	pagebackref,
	linktocpage,
	bookmarksdepth=3
}{hyperref}
\usepackage{hyperref}
\usepackage{bookmark}
\hypersetup{
	citecolor=OliveGreen,
	linkcolor=Mahogany,
	urlcolor=Plum
}

\usepackage[hyphenbreaks]{breakurl}

\newtheorem{theorem}{Theorem}[section]
\newtheorem{corollary}[theorem]{Corollary}
\newtheorem{lemma}[theorem]{Lemma}
\newtheorem{proposition}[theorem]{Proposition}
\newtheorem{question}[theorem]{Question}

\theoremstyle{definition}
\newtheorem{definition}[theorem]{Definition}
\newtheorem{example}[theorem]{Example}

\theoremstyle{remark}
\newtheorem{remark}[theorem]{Remark}

\newtheoremstyle{cited}
{.5\baselineskip plus .2\baselineskip minus .2\baselineskip}
{.5\baselineskip plus .2\baselineskip minus .2\baselineskip}
{\itshape}
{}
{\bfseries}
{.}
{5pt plus 1pt minus 1pt}
{\thmname{#1}\thmnumber{ #2}\thmnote{ \normalfont#3}}

\theoremstyle{cited}

\theoremstyle{remark}

\newtheorem{fact}[theorem]{Fact}

\newtheoremstyle{citeddef}
{.5\baselineskip plus .2\baselineskip minus .2\baselineskip}
{.5\baselineskip plus .2\baselineskip minus .2\baselineskip}
{}
{}
{\bfseries}
{.}
{5pt plus 1pt minus 1pt}
{\thmname{#1}\thmnumber{ #2}\thmnote{ \normalfont#3}}

\theoremstyle{citeddef}

\DeclareMathOperator{\Ext}{Ext}
\DeclareMathOperator{\Tor}{Tor}
\DeclareMathOperator{\Frac}{Frac}

\DeclareMathOperator{\Spec}{Spec}
\DeclareMathOperator{\height}{ht}
\DeclareMathOperator{\depth}{depth}
\DeclareMathOperator{\id}{id}
\DeclareMathOperator{\pd}{pd}

\DeclareMathOperator{\Deg}{deg}
\DeclareMathOperator{\NBIM}{NBIM}
\newcommand{\up}[1]{{{}^{#1}\!}}

\newcommand{\W}{\textbf{W}}

\newcommand{\Cl}{\operatorname{Cl}}
\newcommand{\fm}{\mathfrak{m}}
\newcommand{\fn}{\mathfrak{n}}
\newcommand{\fp}{\mathfrak{p}}

\newcommand{\perf}{\infty}
\newcommand{\hooklongrightarrow}
{\lhook\joinrel\longrightarrow}

\makeatletter
\def\l@subsection{\@tocline{2}{0pt}{2pc}{6pc}{}}
\makeatother

\begin{document}
	
	\title[Homology of Infinite Integral Extensions]{Remarks on some Homological Problems regarding Infinite Integral Extensions}
	
	\author{Mohsen Asgharzadeh and Shravan Patankar}
	
	\address{Mohsen Asgharzadeh, Hakimiyeh, Tehran, Iran.}
	\email{\href{mohsenasgharzadeh@gmail.com}{mohsenasgharzadeh@gmail.com}}
	
	\address{Shravan Patankar, Farmington Hills, Michigan, USA.}
	\email{\href{shravan.patankar@gmail.com}{shravan.patankar@gmail.com}}

	

	\begin{abstract}
Let $R$ be an excellent local domain. $R$ is said to be $NBIM$ if $Tor_{i}^{R}(R^{+}, k) = 0$ for some $i\geq d:=\dim(R)$. Bhatt, Iyengar, and Ma ask if equi-characteristic zero $NBIM$ rings are regular. If $R$ is of positive characteristic, Asgharzadeh and Mahdavi conjecture that $Ext^{i}_{R}(k,R^{\infty}) = 0$ for some $i>d$ implies that $R$ is regular. It is an open question whether $R^{+}$ and $R^{\infty}$ are $\mathfrak{m}$-adically idealwise separated in positive characteristic, a condition from the `local criterion of flatness'. These are analogues of Kunz's theorem and intimately related to the homological conjectures and singularities in algebraic geometry.
We apply a result of Avramov, Hochster, Iyengar, and Yao on contracting endomorphisms to make progress on the first two. We observe that it implies toric $NBIM$ rings are regular and solves the conjecture for $F$-pure rings. These improvements are inaccessible by previous techniques and give new and simple proofs of earlier results. In mixed characteristic, we show several linked results for perfectoid-pure rings. We show the third statement when there is $R\rightarrow S$ finite and flat on the punctured spectrum and $S$ is regular, this uses Cohen-Macaulayness of $S^{+}$.
	\end{abstract}

	\maketitle
	
	\tableofcontents
	
	\section{Introduction}
	In characteristic $p>0$, the \emph{perfect closure} of a commutative ring $R$, denoted by $R^\infty$, adjoins all $p$-power roots to make the Frobenius map bijective, while the \emph{absolute integral closure} $R^+$ adjoins roots of all monic polynomials and is the maximal integrally closed extension. The systematic use of these closures was pioneered by Artin~\cite{Art}, and Hochster and Huneke~\cite{HH} famously exploited them as vast ``big'' Cohen--Macaulay algebras. More recently, Bhatt \cite{bat} has advanced these ideas through his work on big Cohen--Macaulay algebras and perfectoid techniques, extending their reach to the mixed characteristic setting.
	
	By contrast, in equal characteristic zero, $R^+$ is far less studied,  as the subtle inseparability phenomena driving positive- and mixed-characteristic developments do not arise. Motivated by a question of Bhatt--Iyengar--Ma~\cite{bhatt2019Regular} on homological vanishing over large integral extensions, we investigate---in the equicharacteristic zero case and, more generally, together with their dual counterparts in positive characteristic---whether vanishing of (co)homology over extensions such as $R^+$ and $R^\infty$ can detect the regularity of a Noetherian local domain $(R,\mathfrak m,k)$. More precisely, our work is guided by the following problems.
	\begin{itemize}
		\item[$(Q_1)$] \label{question:BIMquestion}  (See \cite[End of \S4, Question]{bhatt2019Regular}). Suppose \(R\) is of  equicharacteristic zero. If
		\(
		\Tor_i^R(k,R^+)=0
		\)
		for some \(i\ge 1\), must \(R\) be regular?
		
		\item[(C)] (See \cite{Elham}). Suppose \(k\) is  of characteristic \(p>0\) and $P$ is a perfectoid $R$-algebra. If
		\(
		\Ext_R^i(k,P)=0
		\)
		for some \(i>\dim(R)\), then \(R\) is regular.
		\item[$(Q_2)$] (See \cite[Introduction]{Elham}). If \(R\) is   of equicharacteristic zero and \(\Ext_R^i(k,R^{+}) = 0\) for some \(i \geq \dim(R)+1\), then is \(R\) 
		regular?
		\item[$(Q_3)$] (See \cite[Section 3]{Shravan}). Suppose $R$ 
		is excellent and of characteristic $p>0$. Are
		$R^+$ and $R^\infty$ $\fm$-adically idealwise separated as $R$-modules?
	\end{itemize}
	
	These questions serve as the starting point of our investigation into toric rings and contracting endomorphisms, where such vanishing phenomena can be analyzed explicitly.
	Indeed, Section~3 is divided into three subsections devoted, respectively, to the equicharacteristic zero, positive characteristic, and mixed characteristic settings. We begin with the following theorem.
	
	\begin{theorem}\label{main1}(See Theorem \ref{3.2}).
		Let \(R\) be a normal toric ring containing \(\mathbb{Q}\). Then the following assertions hold.
		\begin{itemize}
			\item[(a)] If \(\Tor_i^R(k,R^+)=0\) for some \(i>0\), then \(R\) is regular.
			
			\item[(b)] \(\Ext_R^i(k,R^+)=0\) for some \(i>\dim(R)\) if and only if \(R\) is regular.
			
			\item[(c)] \(\operatorname{Gid}_R(R^+)<\infty\) if and only if \(R\) is Gorenstein.
			
			\item[(d)] \(\operatorname{Gpd}_R(R^+)<\infty\) if and only if \(R\) is Gorenstein.
		\end{itemize}
	\end{theorem}
	
	We emphasize that item~(a) is not a characterization of regularity for small $i$. Indeed, if \(R\) is Cohen--Macaulay with \(\dim(R)>2\), then
	\(
	\Tor_1^R(k,R^+)\neq 0;
	\)
	see \cite[Theorem~3.2]{Shravan} and \cite[Theorem~1.2~(c)]{Elham}.
	As an application, Theorem~\ref{main1} (a) removes the unnecessary direct summand hypothesis from \cite[Observation~5.7 (iii)]{Elham}. It also immediately yields the following corollary, which unifies and simplifies \cite[Proposition~5.6]{Elham} and \cite[Proposition~2.1]{RSS}, since the rings considered there are normal toric rings.
	
	\begin{corollary}(See corollary~\ref{quad}).
		Let
		\(
		R=k[x,y,z,w]/(xw-yz),
		\)
		where \(k\) is a field of characteristic zero. Then
		\(
		\Tor_i^R(k,R^+)\neq0
		\)
		for every \(i\ge0\).
	\end{corollary}
	Clearly, Theorem~\ref{main1}~(b) confirms $(Q_2)$ for $\NBIM$ rings (see Definition \ref{3.1}).
	Concerning items~(c) and~(d) from Theorem~\ref{main1}, we recall that \(\operatorname{Gid}_R(-)\) and \(\operatorname{Gpd}_R(-)\) stand for the Gorenstein injective and Gorenstein projective dimensions, respectively.
	It is worth recalling that the following question was raised in \cite[Question 6.7]{Asg26}: \begin{center}does the finiteness of
		\(
		\operatorname{Gid}_R(R^+)
		\)
		force \(R\) to be Gorenstein?\end{center} The analogous question for the Gorenstein projective dimension of perfectoid algebras was answered affirmatively in \cite[Corollary 3.4]{Ryo}. To the best of our knowledge, however, no corresponding result is currently known for infinite integral extensions of rings containing \(\mathbb{Q}\).
	However, if $\mathbb{Q}\subseteq R$ is normal, we answer it completely:

	\begin{corollary}(See Corollary \ref{c3.2}).
		Suppose that $R$ is a normal local ring containing $\mathbb{Q}$. Then   $\operatorname{Gid}_R(R^+)<\infty$ iff $R$ is Gorenstein.
	\end{corollary}

	Being a unique factorization domain is a weaker property than regularity. It was shown in \cite[Proposition~5.6]{Elham} that every $\NBIM$ of finite Cohen--Macaulay type is a unique factorization domain. We present another evidence:
	
	\begin{corollary} (See Corollary \ref{tcl}).
		Suppose that \( R \) is a local ring of equicharacteristic zero with torsion class group. If
		\(
		\Tor_i^R(k,R^+)=0,
		\) for some $i>0$,
		then \( R \) is a UFD.
	\end{corollary}

	In positive characteristic, we first establish a regularity criterion over 
	$F$-pure rings (Theorem~\ref{pur}), which improves upon 
	\cite[Proposition 5.15]{Elham} by eliminating the isolated singularity 
	hypothesis. This result is then extended to mixed characteristic in 
	Proposition~\ref{mainregularity}. Under the additional assumption of an isolated 
	singularity, we obtain a strengthened formulation, yielding the following 
	mixed-characteristic analogue of \cite[Proposition 5.17]{Elham}:
	
	\begin{corollary}[See Corollary~\ref{isom}]\label{mixedpure}
		Let $(R, \mathfrak{m})$ be a mixed-characteristic local ring with an isolated 
		singularity. Suppose there exists a perfectoid $R$-algebra $P$ with 
		$P \neq \mathfrak{m}P$ and an integer $i > \dim(R)$ such that
		\(\Ext_R^i(k, P) = 0.\)
		Then $R$ is regular.
	\end{corollary}

	Our main theorem in \S4 is as follows, which confirms $(Q_3)$ in a particular case.
	
	\begin{theorem}[See Theorem~\ref{thm:regular-cover-idealwise-separated}]
		Let $(R,\fm)$ be an excellent local domain of characteristic $p>0$. Suppose
		that there is a module-finite regular local extension
		$S$ of $R$
		such that $S/R$ is flat on the punctured spectrum of $R$.
		Then $R^+$ is $\fm$-adically idealwise separated as an $R$-module.
	\end{theorem}
	\medskip
	\noindent
	The following question by Shimomoto remains elusive, slightly rephrased from its original version: 
	\begin{question}[{\cite[Question~4]{Shi10}}]
		Let $R$ be a domain of positive characteristic. If $R^{\infty}$ is coherent, must $R$ be a purely inseparable extension of a regular ring?
	\end{question}
	
	To our knowledge, techniques below do not provide any significant insights. The answer is not known in general even if one assumes $R$ is a Gorenstein UFD strongly $F$-regular complete local domain. For related results see \cite{Pat}, \cite{Simp}, \cite{AB} and \cite{moh}.
	
	\subsection{Acknowledgments and AI disclosure}

	We thank Elham Mahdavi, Kazuma Shimomoto, Shiji Lyu, and Sridhar Venkatesh for conversations and constructive comments. The main results of this paper were conceived after progress on Question $(Q_1)$ in \cite{Elham} and we would like to thank the authors for their continued interest.  
	
	The second author used a publicly available LLM to assist in finding references and for background proofs of routine facts from algebraic geometry, such as existence of endomorphisms on toric varieties. No AI was used in proofs or discovery of the main results. All writing in this paper is entirely produced by the human authors based upon their own understanding and verification. See the Leiden Declaration on Artificial Intelligence and Mathematics \cite{Leiden} for additional guidance.

	\section{Preliminaries}
	Let $(R,\mathfrak m,k)$ be a Noetherian  local or graded ring.
	The notation \(\pd_R(-)\) (resp. \(\id_R(-)\)) stands for the projective (resp. injective) dimension of \((-)\).
	The absolute integral closure of an integral domain \(R\), denoted by \(R^{+}\), is the integral closure of \(R\) inside an algebraic closure of its fraction field.
	
	
	Suppose now that $R$ is of prime characteristic $p>0$, and
	$f: R \to R$
	denotes the Frobenius endomorphism, given by $f(a)=a^p$ for every $a\in R$. For each integer $n\geq 0$, the $n$th iterate $f^n$ induces a new $R$-module structure on the underlying abelian group of $R$. We denote this module by ${}^{f^n}\!R$, where the scalar multiplication is defined by
	\(
	a\cdot b = a^{p^n} b,
	\)
	for $a,b\in R$.
	Recall that
	\[
	R^\infty := \varinjlim \bigl( R \xrightarrow{f} R \xrightarrow{f} R \xrightarrow{f} \cdots \bigr).
	\]
	
	A natural extension of Frobenius beyond rings of prime characteristic is as follows.
	
	\begin{definition}\label{contd}
		A local ring homomorphism \( \phi: (R, \mathfrak{m}) \to (R, \mathfrak{m}) \) is called a \emph{contraction} if, for every \( r \in \mathfrak{m} \), the sequence \( (\phi^i(r))_{i \geq 1} \) converges to zero in the \( \mathfrak{m} \)-adic topology of \( R \).
	\end{definition}
	
	For the convenience of the reader, we recall the main result of Avramov, Hochster, Iyengar, and Yao (\cite[Theorem~1.1]{AHIY}):
	
	\begin{theorem}\label{AHIY} \cite[Theorem 1.1]{AHIY}
		Let \(R\) be a local ring, and let \(\varphi\colon R \to R \) be a contracting endomorphism. If there exist a nonzero finite \(R\)-module \(M\) and an integer \(i\geq 1\) such that the \(R\)-module \({}^{\varphi^i}\!M\) has finite flat dimension or finite injective dimension, then \(R\) is regular.
	\end{theorem}
	
	
	\begin{definition}\label{def:flat-codimension}
		Let $A$ be a Noetherian ring and let $N$ be an $A$-module. We say that $N$ is
		\emph{flat in codimension one} if $N_{\fp}$ is flat over $A_{\fp}$ for every
		$\fp\in\Spec(A)$ with $\height(\fp)\leq 1$.
		If $(A,\mathfrak m)$ is local, we say that $N$ is
		\emph{flat on the punctured spectrum} if $N_{\fp}$ is flat over $A_{\fp}$ for
		every $\fp\neq\mathfrak m$.
	\end{definition}

	\begin{definition}\label{def:pure-map}
		An injective homomorphism of $R$-modules
		\(
		N\to M
		\)
		is called \emph{pure} if
		\(
		L\otimes_RN\to L\otimes_RM
		\)
		is injective for every $R$-module $L$. It is called
		\emph{cyclically pure} if
		\(
		(R/I)\otimes_RN\to(R/I)\otimes_RM
		\)
		is injective for every ideal $I\subseteq R$.
	\end{definition}
	By $R$ is $F$-pure, we mean $f: R \to R$ is pure.

	
	We recall several notions related to separatedness. We emphasize that these are closely related to the notions of Ohm--Rush, Ohm--Rush trace, Mittag--Leffler, and intersection flatness. See \cite{DET}, \cite{EpsteinShapiro}, and \cite{HJ} for related definitions and results.
	
	\begin{definition}\label{def:separation-notions}
		Let $\mathfrak a\subseteq R$ be an ideal.
		\begin{itemize}
			\item[(a)] The module $M$ is said to be
			\emph{$\mathfrak a$-adically separated} if
			\(
			\bigcap_{n\geq 1}\mathfrak a^nM=0.
			\)
			
			\item[(b)] The module $M$ is said to be
			\emph{$\mathfrak a$-adically idealwise separated} if
			\(
			\bigcap_{n\geq1}
			\mathfrak a^n\bigl(I\otimes_RM\bigr)=0
			\)
			for every ideal $I\subseteq R$.
			
			\item[(c)] Following \cite{Zoeschinger}, the module $M$ is said to be
			\emph{totally $\mathfrak a$-adically separated} if
			\(
			\bigcap_{n\geq1}
			\mathfrak a^n\bigl(L\otimes_RM\bigr)=0
			\)
			for every finitely generated $R$-module $L$.
			When $\mathfrak a=\fm$, we simply say that $M$ is separated,
			idealwise separated, or totally separated, respectively.
		\end{itemize}
	\end{definition}
	
	Since every ideal of a Noetherian ring is finitely generated, total
	separatedness immediately implies idealwise separatedness.

	\begin{remark}\label{rem:local-flatness-idealwise}
		Most results in this paper can be generalized from $\mathfrak a$-adically idealwise separated to totally $\mathfrak a$-adically separated. We leave this to the reader. Some content about separatedness in this section overlaps with second author's PhD thesis \cite{Shravan} and \cite{Pat}.
	\end{remark}
	
	\begin{remark}\label{rem:local-flatness-idealwise}
		The significance of Definition~\ref{def:separation-notions} comes from the local
		criterion of flatness. Indeed, let $(R,\fm,k)$ be a Noetherian local
		ring and let $M$ be an $\fm$-adically idealwise separated $R$-module.
		Since $M/\mathfrak \fm M$ is automatically flat over the field
		$R/\fm$, Matsumura's local criterion of flatness
		\cite[Theorem~22.3]{mat} gives
		\[
		\Tor_1^R(k,M)=0
		\quad\Longrightarrow\quad
		M\text{ is flat over }R.
		\]
	\end{remark}

	
	\begin{proposition}\label{prop:flat-separated-idealwise}
		Let $(R,\fm)$ be a local ring and let $M$ be a flat, $\fm$-adically
		separated $R$-module. Then $M$ is $\fm$-adically idealwise separated.
	\end{proposition}
	
	\begin{proof}
		Let $I\subseteq R$ be an ideal. Since $M$ is flat, tensoring the
		injection $I\hookrightarrow R$ with $M$ gives an injection
		\(
		I\otimes_RM\hooklongrightarrow R\otimes_RM\cong M.
		\)
		Its image is $IM$, so we may identify
		\(
		I\otimes_RM\cong IM\subseteq M.
		\)
		Let
		\(
		\alpha\in
		\bigcap_{n\geq1}
		\fm^n\bigl(I\otimes_RM\bigr).
		\)
		Under the preceding injection, the image of $\alpha$ belongs to
		$\fm^nM$ for every $n$, since
		\(
		\fm^n\bigl(I\otimes_RM\bigr)
		\to
		\fm^nM.
		\)
		Therefore the image of $\alpha$ belongs to
		\(
		\bigcap_{n\geq1}\fm^nM=0.
		\)
		The map $I\otimes_RM\to M$ is injective, so $\alpha=0$.
		Thus $I\otimes_RM$ is $\fm$-adically separated for every ideal
		$I\subseteq R$.
	\end{proof}

	\begin{remark}\label{cor:flat-complete-total}
		Let $(R,\fm)$ be a complete Noetherian local ring and let $M$ be a
		flat, $\fm$-adically complete $R$-module. Then $M$ is totally
		$\fm$-adically separated, Ohm-Rush, Ohm-Rush trace, Mittag--Leffler,
		and intersection flat.
	\end{remark}
	
	\begin{proof}
		The result follows from \cite[Corollary~4.3.14]{DET}.
	\end{proof}

	\section{Homological properties of infinite integral extensions}

	In what follows, we assume some familiarity with Gorenstein homological algebra, see \cite{Totally}
	for a nice survey. In particular, 
	$\operatorname{Gid}_R(-)$ (resp.~$\operatorname{Gpd}_R(-)$) stands for Gorenstein injective dimension (resp.~Gorenstein projective dimension).
	\subsection{Rings containing \(\mathbb{Q}\)}
	
	We begin with the following definition.
	
	\begin{definition}\label{3.1}
		We say that \(R\) is \textit{NBIM} (abbreviated for Normalized Bhatt--Iyengar--Ma) if
		\(
		\operatorname{Tor}_i^R(k, R^+) = 0
		\)
		for some \(i > \dim(R)\).
	\end{definition}
	
	\begin{remark}
		Concerning {Definition} \ref{contd}, 
		as \( R \) is Noetherian, it is equivalent to the existence of some integer \( i > 1 \) such that \( \phi^i(\mathfrak{m}) \subseteq \mathfrak{m}^2 \), since it suffices to check the condition on a finite generating set of \( \mathfrak{m} \).
		The notion of a contraction extends naturally to the \( \ast \)-local graded setting. Indeed, let \( R := \bigoplus_{\alpha \in \mathbb{N}_0^t} R_\alpha \) be graded equipped with a unique maximal homogeneous ideal
		\(
		R_+=\bigoplus_{\alpha\neq 0}R_\alpha.
		\) and let \( \phi: R \to R \) be a homogeneous ring homomorphism. Then \( \phi \) is called a contraction if 
		\(
		\phi^i(R_+) \subseteq (R_+)^2
		\)
		for some integer \( i > 1 \).
	\end{remark}
	
	\begin{remark}\label{admis}
		In order to use the results on $\mathbb{Z}$-graded rings for affine
		semigroup rings, \cite[Page 260]{BH} indicates that a decomposition
		of the $k$-vector space $k[C]=\bigoplus_{t\in\mathbb{N}_0}k[C]_t$ is an \emph{admissible grading} if $k[C]$ is a positively
		graded $k$-algebra with respect to this decomposition, and furthermore
		each component $k[C]_t$ is a direct sum of finitely many $\mathbb{Z}^C$-graded components. 
	\end{remark}
	The following seems to be well-known.
	\begin{fact} 
		Let \(R=\bigoplus_{\alpha\in S}R_\alpha\) be a finitely generated \(S\)-graded \(k\)-algebra, where \(S\) is a positive affine semigroup and \(R_0=k\). Put
		\(
		\mathfrak m=R_+.
		\) Then:
		\begin{itemize}
			\item[(a)]  $R$  is regular if  $R_{\mathfrak m}$   is regular.
			\item[(b)] $R$  is Cohen-Macaulay if  $R_{\mathfrak m}$   is Cohen-Macaulay.  \item[(c)] If, in addition, \(R\) is Cohen--Macaulay and admits a graded canonical module \(\omega_R\), then
			$R_{\mathfrak m}$  is Gorenstein implies that
			R is Gorenstein.
		\end{itemize}
	\end{fact}
	
	\begin{proof}
		Indeed, the desired claims are in \cite[Ex. 2.2.24~(c)]{BH}, \cite[Ex. 2.1.27~(c)]{BH}
		and \cite[Ex. 3.6.20~(c)]{BH} in the $\mathbb{Z}$-graded case. It remains to apply {Remark} \ref{admis}, and recall from \cite[Proposition 6.1.5]{BH} that
		any  positive affine semigroup  furnished with an admissible
		grading.
	\end{proof}

	Although the presentation has been modified, the following gives the proof of ~\texorpdfstring{Theorem~\ref{main1}}{Theorem 1}.
	
	\begin{theorem}\label{3.2}
		Suppose that $R$ is a normal toric ring containing $\mathbb{Q}$, (or more generally having a contracting endomorphism). The following holds:
		\begin{itemize}
			\item[(a)] $R$ \text{is}   $\NBIM$ iff $R$ is regular \footnote{The normality assumption may be omitted, as it follows from Tor-vanishing (see \cite[Proposition 4.2]{Shravan}).}
			\item[(b)] 
			\(
			\Ext^i_R(k,R^+)=0
			\)
			for some $i>\dim(R)$ iff $R$ is regular.
			
			\item[(c)]   $\operatorname{Gid}_R(R^+)<\infty$ iff $R$ is Gorenstein.
			
			\item[(d)]   $\operatorname{Gpd}_R(R^+)<\infty$ iff $R$ is Gorenstein.
		\end{itemize} 
	\end{theorem}
	
	\begin{proof} 
		Let $n\in \mathbb{N}$ be a fixed integer, and assume 
		$n
		\geq 2$. Let $B$ be a semi-group so that $R=k[B]$. It turns out that we can additionally assume 
		$B$ is positive as an affine semigroup (cf. \cite[Proposition 6.1.3]{BH} (a), (b)). The assignment $b\mapsto b^n$ induces a contracting homomorphism.
		It follows that $\up{\varphi}R$  is finitely
		generated as an $R$-module. 
		Since \(R\) is normal, every finite extension of \(R\) splits (see \cite[Lemma 2.1]{Shravan}).
		In other words, there is a factorization 
		\( 
		R \to\up{\varphi}R  \to R^+
		\)
		whose composition is the natural inclusion. Consequently,
		$\up{\varphi}R $ is an $R$-direct summand of $R^+$.
		The ``if" parts of all items are clear, so we present only the proof of the ``only if" direction.

		(a):  We have
		\(
		\operatorname{Tor}^R_i(k, \up{\varphi}R )=0.
		\) 
		so that
		$ 
		\operatorname{pd}_R(\up{\varphi}R )<\infty.
		$ 
		Finally, the regularity  criterion given in 
		Theorem \ref{AHIY} implies that $R_{\fm}$ is regular, where $\fm$
		is the unique maximal homogeneous ideal. Then 
		$R$ is regular.
		
		(b): We have
		\(
		\Ext^i_R(k, \up{\varphi}R )=0.
		\) Recall  that $\up{\varphi}R$  is finitely
		generated as an $R$-module. By \cite[2.2]{Ro2}
		$\mu ^i(\up{\varphi}R)\neq  0$ if and only if $i\in[\depth(\up{\varphi}R),\id(\up{\varphi}R)]$. Thus,
		$ 
		\operatorname{id}_R(\up{\varphi}R )<\infty
		$.
		Theorem \ref{AHIY} implies that $R_{\fm}$ is regular, and so
		$R$ is regular.

		(c): Since $R$ is a homomorphic image of a Gorenstein
		ring of finite Krull dimension, and 
		in view of \cite[Proposition 3.20]{Totally}, we observe that $\operatorname{Gid}_{R_{\fm}}((R^+)_{\fm})<\infty$, 
		where $\fm$
		is the unique maximal homogeneous ideal. Set $A:=R_{\fm}$. This is also equipped with a contraction map, again denoted by $\varphi$. Then
		$\operatorname{Gid}_{A}(A^+)<\infty$.
		Let $E$ be an injective $A$-module. Then, by definition, $\Ext^j_A(E,A^+)=0$ for all $j\gg 0$ (see \cite[Theorem 3.6]{Totally}). We apply this along with the first part to see  $\Ext^j_A(E,\up{\varphi}A )=0$ for all $j\gg 0$.  Then $\operatorname{Gid}_A(\up{\varphi}A)<\infty$.
		Recall that $A$ is a homomorphic image of a Gorenstein ring. This allows us to apply  \cite[Theorem 5.5]{FF} and deduce that
		$A=R_{\fm}$ is Gorenstein. Then, by the above fact,
		$R$  is Gorenstein.

		(d): This case is dual to (c), and the requisite straightforward modification is left to the reader.
	\end{proof}

	The following corollary unifies and streamlines both \cite[Proposition 5.6]{Elham} and \cite[2.1]{RSS}, as the ring in each of those settings is normal and toric. Note that this ring is a Gorenstein rational singularity with an infinite divisor class group. 
	
	\begin{corollary}\label{quad}
		Let \( R = k[x,y,z,w]/(xw - yz) \), with \( k \) a field of characteristic zero. Then \( \Tor_i^R(k, R^+) \neq 0 \) for all \( i \geq 0 \).
	\end{corollary}
	
	\begin{remark}
		In \cite{RSS}, Roberts, Singh, and Srinivas show annihilation of local cohomology modules by elements of arbitrarily small order using the multiplication map on Abelian varieties. Then Albanese map is used to pull back (along the multiplication map) to get finite covers. It might be possible to obtain similar cohomology vanishing via pulling back along the contracting endomorphisms on toric varieties. A full exploration is outside the scope of this project.
	\end{remark}
	\begin{corollary}\label{c3.2}
		Suppose that $R$ is a normal local ring containing $\mathbb{Q}$. Then   $\operatorname{Gid}_R(R^+)<\infty$ iff $R$ is Gorenstein.
	\end{corollary}
	
	\begin{proof} 
		Suppose $\operatorname{Gid}_R(R^+)<\infty$. Since $R$ is normal and 
		containing $\mathbb{Q}$, it turns out that $R$ is a directed summand of 
		$R^+$ (see \cite[Lemma 2.1]{Shravan}). By \cite[Proposition 3.3]{Totally} we deduce that
		$\operatorname{Gid}_R(R)<\infty$. By \cite[Proposition 3.3]{Totally}
		$\operatorname{id}_R(R)=\operatorname{Gid}_R(R)<\infty$, as $R$ is projective, and so $R$ is Gorenstein.
	\end{proof}
	\begin{question} Thanks to  Corollary \ref{c3.2} and \cite[Question 6.7]{Asg26}, we ask:
		\begin{itemize}
			\item[(a)] Suppose that $R$ is a normal local ring containing $\mathbb{Q}$ and  $\operatorname{Gpd}_R(R^+)<\infty$. Is $R$  Gorenstein?
			
			\item[(b)] Is the normality assumption really needed?\end{itemize}
	\end{question}
	As observed in \cite{Shravan}, a question of Andr\'e and Fiorot \cite{AF} on fpqc analogues of splinters is related to the question of Bhatt, Iyengar, and Ma. It is likely that our results have some implications towards it as well (with suitable assumptions on the algebras, as noted in \cite{Shravan}), however we leave precise formulations to the reader.

	\begin{lemma}(See also \cite[Proposition 4.2]{Sing}).
		Let $R$ be a normal domain, and let $I$ be a divisorial ideal whose class has
		finite order $n$ in the divisor class group $\Cl(R)$. Suppose
		\(
		I^{(n)}=uR
		\)
		for some $u\in R$, and let
		\(
		S=R[It,\ldots,I^{(i)}t^i,\ldots]/(ut^n-1)
		\)
		be the associated cyclic cover. Then $S$ admits a natural
		$\mathbb{Z}/n\mathbb{Z}$-grading. Moreover,
		\(
		S=\bigoplus_{i=0}^{n-1} I^{(i)}t^i
		\)
		as a graded ring, and
		as $R$-modules. In particular, $I$ is an $R$-module direct summand of $S$.
	\end{lemma}
	
	\begin{proof}
		Endow $R$ with the trivial grading by $\mathbb{Z}/n\mathbb{Z}$, so that every
		element of $R$ has degree $0$, and assign
		\(
		\deg(t)=1\in\mathbb{Z}/n\mathbb{Z}.
		\)
		Since $u\in R$, we have $\deg(u)=0$, and hence
		\(
		\deg(ut^n)=0+n=0
		\)
		in $\mathbb{Z}/n\mathbb{Z}$. Therefore the defining relation $ut^n-1$ is
		homogeneous, and $S$ inherits a natural $\mathbb{Z}/n\mathbb{Z}$-grading.
		For $0\le i\le n-1$, the homogeneous component of degree $i$ is
		\(
		S_i=I^{(i)}t^i,
		\)
		where $I^{(0)}=R$. Since the relation $ut^n=1$ identifies $t^n$ with the unit
		$u^{-1}$, every element of $S$ can be uniquely written as
		\(
		\sum_{i=0}^{n-1}x_it^i\), where \( x_i\in I^{(i)}.
		\)
		Hence
		\(
		S=\bigoplus_{i=0}^{n-1}I^{(i)}t^i
		\)
		as a graded ring.
		Finally, multiplication by $t^i$ induces an $R$-module isomorphism
		\(
		I^{(i)}\xrightarrow{\ \cong\ } I^{(i)}t^i\) via
		\(x\mapsto xt^i.
		\)
		Therefore
		\(
		S\cong R\oplus I\oplus\cdots\oplus I^{(n-1)}
		\)
		as $R$-modules. In particular, the degree-one component
		\(
		It\cong I
		\)
		is an $R$-module direct summand of $S$.
	\end{proof}

	Being a unique factorization domain is a weaker property than regularity. It was shown in \cite[Proposition~5.6]{Elham} that every  $\NBIM$  of finite Cohen--Macaulay type is a unique factorization domain.
	Here, is another sample.

	\begin{corollary}
		\label{tcl}
		Suppose that \( R \) is a local (or \( \mathbb{N} \)-graded) ring of equicharacteristic zero with torsion class group. If
		\(
		\Tor_i^R(k,R^+)=0,
		\) for some $i>0$,
		then \( R \) is a UFD.
		
	\end{corollary}
	
	\begin{proof}
		We know $R$ is normal.  
		Assume, for contradiction, that there exists a non-principal ideal \( I \in \Cl(R) \). Let \( n >1\) be its order, so that \( I^{(n)} = uR \) for some (homogeneous) element \( u \in R \). Let $S$ be the cyclic cover 
		Since the extension is finite, \( R \subseteq S \subseteq R^+ = S^+ \). Recalling that \( R \) is a splinter, we deduce that \( \Tor_i^R(k,S) = 0 \). 
		In particular, \( p := \operatorname{pd}_R(S) < \infty \). By previous lemma \( I \cong It \) is a direct summand of a module of finite projective dimension, it follows that \( \operatorname{pd}_R(I) < \infty \). As \( R \) is normal and \( \height(I) = 1 \), Kaplansky's trick implies that \( I \) is principal, a contradiction. 
	\end{proof}
	
	See \cite{Tavanfar} and \cite[Corollary 4.1]{Tavanfar} for related results. We also note an ascent statement, with the hope of making progress towards Question \ref{question:BIMquestion}.
	
	\begin{proposition}\label{prop:nbim-etale-ascent}
		Let
		\( 
		(R,\fm,k)\to(S,\fn,\ell)
		\)
		be a finite, local, and \(\acute{e}\)tale homomorphism of excellent normal
		local domains. If $R$ is $\NBIM$, then $S$ is $\NBIM$.
	\end{proposition}
	
	\begin{proof}
		We have $\ell \simeq S\otimes_{R} k$ and $S^{+} \simeq R^{+}$ and hence $$\Tor_{j}^{S}(S\otimes_{R}k,S^{+})\simeq \Tor_{j}^{R}(k,S^{+})=\Tor_{j}^{R}(k,R^{+})=0,$$
		as desired.\end{proof}
	
	\begin{remark}
		It would be interesting to know more permanence properties of NBIM rings, behavior under ultraproducts, quasi-\(\acute{e}\)tale extensions, and so on.
	\end{remark}

	\subsection{Rings containing \texorpdfstring{$\mathbb{F}_p$}{Fp}}
	
	\begin{remark}The corresponding Ext-statement of \S3.1 cannot be obtained by simply dualizing the Bhatt--Iyengar-Ma statement. Although Matlis duality exchanges \(\operatorname{Ext}\) and \(\operatorname{Tor}\) via
		\(
		\operatorname{Ext}_R^i(k, R^+)^\vee \cong \operatorname{Tor}_i^R(k, (R^+)^\vee),
		\)
		the issue is that \((R^+)^\vee \ncong R^+\), e.g., \cite[Corollary 4.12]{bhatt2019Regular} says that
		$H^+_{\fm}((R^+)^\vee)=0$ but  $H^d_{\fm}(R^+)\neq0$ at least in prime characteristic and $d>0$. Consequently, the question \((Q_1)\) from the introduction is dual to \((Q_2)\), but with the wrong module in the wrong position. Moreover, their full general statements (\cite[Theorem 2.1, Corollary 4.8]{bhatt2019Regular}) have a $JU \neq U$ assumption, which generally \emph{fails} for \((R^+)^\vee\) or modules involving injective hulls.\end{remark}

	\begin{proposition}\label{pur}
		Let \(R\) be F-pure. If 
		\(
		\operatorname{Ext}_R^i(k, R^\infty) = 0
		\)
		for some \(i > \dim(R)\), then \(R\) is regular.
	\end{proposition}
	
	\begin{proof}
		First, assume that \(R\) is F-finite. Recall from the purity of the extension \({}^f\!R \subseteq R^\infty\) that
		\(
		\operatorname{Ext}_R^i(k, {}^f\!R) \subseteq \operatorname{Ext}_R^i(k, R^\infty) = 0.
		\)
		Hence \(\operatorname{Ext}_R^i(k, {}^f\!R) = 0\). Since \({}^f\!R\) is finite over \(R\), and by \cite{Ro2} Bass numbers are nonzero only in the range 
		\(
		[\depth({}^f\!R), \, \operatorname{id}_R({}^f\!R)],
		\)
		we deduce that \(\operatorname{id}_R({}^f\!R)\) is finite. It then follows   that \(R\) is regular (e.g.  see \cite{AHIY}).
		To remove the F-finite hypothesis, recall that there exists a local ring extension \((S, \mathfrak{n})\) of \((R, \mathfrak{m})\) such that \(\mathfrak{m}S = \mathfrak{n}\), \(S\) is F-finite, \(S\) is faithfully flat over \(R\), and \(S\) has infinite residue field. For instance, if \(\widehat{R} \cong k[[x_1,\dots,x_m]]/I\) for some ideal \(I\), we may take 
		\[
		S = \overline{k}[[x_1,\dots,x_m]] / I\overline{k}[[x_1,\dots,x_m]],
		\]
		where \(\overline{k}\) denotes the algebraic closure of \(k\).
		Now, if \(M\) is an \(R\)-module, then
		\[
		(M \otimes_R {}^f\!R) \otimes_R S \cong (M \otimes_R S) \otimes_S {}^f\!S.
		\]
		Moreover, \(\Ext_R^i(M, -)\) commutes with flat base change whenever \(M\) is finitely generated over \(R\) (see \cite[Ex. 7.7]{mat}). Applying these observations, together with the F-finite case already established, we conclude that \(S\) is regular. Finally, by \cite[Theorem 2.2.12]{BH}, regularity descends along faithfully flat extensions, so \(R\) is regular as well.
	\end{proof}

	The following extends both Proposition \ref{pur} and \cite[5.17, 5.18]{Elham} proved by Mahdavi.
	\begin{corollary}\label{purp}
		Suppose  
		\(
		\operatorname{Ext}_R^i(k(\fp), (R_{\fp})^\infty) = 0
		\)
		for some \(i \geq 2\dim(R)\) and all $\fp$ in F-pure locus of $R$, then \(R\) is regular.
	\end{corollary}
	
	\begin{proof} 
		By {Proposition}~\ref{pur} $R_p$ is regular. It follows that \(
		\operatorname{Ext}_R^i(k(\fp), (R_{\fp})^\infty) = 0
		\) for all $\fp$ in singular locus of $R$. 
		In particular,  \(
		\operatorname{Ext}_R^j(k(\fp), (R_{\fp})^\infty) = 0
		\) for all $j>i,$  (see (\cite[Theorem 1.3]{IYENGAR}). These   enable  us to use \cite[Theorem VI.9]{sc3}, and deduce that \(R^\infty\) has finite injective dimension. 
		By \cite[Example 3.17(i)]{Ryo} the ring $R$ is regular as claimed.  
	\end{proof}\begin{corollary}\label{purp2}
		Suppose  
		\(
		\operatorname{Ext}_{R_{\fp}}^i(k(\fp), (R_{\fp})^\infty) = 0
		\)
		for some \(i \geq  \dim(R)\) and all $\fp$ in F-pure locus of $R$, then \(R\) is regular.
	\end{corollary}
	Suppose \(\id_R(R^\infty)<\infty\), and recall from 
	\cite[Example 3.17(i)]{Ryo} that \(R\) is regular. Here, we study the deformation of this property.
	
	\begin{proposition}
		Let \(R\) be a Cohen-Macaulay  ring and let  $\underline{x}:=x_1,\ldots,x_d$ be a system of parameter. If 
		\(
		\operatorname{Ext}_R^i(k, R^\infty) = 0
		\)
		for some \(i > \dim(R)\), then \( \id_{{R}}(R^\infty/\underline{x} R^\infty)<\infty\).
	\end{proposition}
	
	\begin{proof}
		Recall that $\underline{x}$ is both $R$-regular and $R^\infty$-regular. Set $i':=i-d>0$. One has 
		$$
		0=\operatorname{Ext}_R^i(k, R^\infty)\cong  \operatorname{Ext}_{\overline{R}}^{i'}(k, R^\infty/\underline{x} R^\infty),$$(see \cite[Theorem 10.77]{Rot}),
		where $\overline{R}:=R^\infty/\underline{x} R^\infty$. The module $\overline{R}$ is $\fm_{\overline{R}}$-torsion. Then its Matlis dual, taken over $\overline{R}$, is complete (see~\cite[\S 4.2, Lemma]{sim}). So, 
		$$0=\operatorname{Ext}_{\overline{R}}^{i'}(k, R^\infty/\underline{x} R^\infty)^{\vee_{\overline{R}}}=\Tor^{\overline{R}}_{i'}(k, (R^\infty/\underline{x} R^\infty)^{\vee_{\overline{R}}}),$$(see  \cite[\S 4.1]{sim}).
		By \cite[Proposition 2.1]{sim} we deduce that $\operatorname{fl.dim}_{\overline{R}}((R^\infty/\underline{x} R^\infty)^{\vee_{\overline{R}}})<\infty$,  and so $\id_{\overline{R}}(R^\infty/\underline{x} R^\infty)<\infty$. This implies that $\id_{R}(R^\infty/\underline{x} R^\infty) <\infty$, by the first theorem on injective change of rings. 
	\end{proof}
	
	In the final version of  \cite{Asg26}, conjecture (C) has been solved when $\operatorname{char}(R)$ —after the writing of this paper and by an approach different from the one employed here.

	\subsection{Rings of mixed characteristic}
	Here, is a mixed version of {Theorem} \ref{pur}  and recall from  \cite[Definition 4.1]{ket} that, 
	with $p$ in its Jacobson radical, $R$ is called
	\textit{perfectoid-pure}  if there exists a perfectoid $R$-algebra $P$ such that $R\to P$ is pure. For being perfectoid in addition to \cite{ket} we assume $\fm P\neq P$.
	\begin{proposition}\label{m515}
		Let \((R,\fm)\) be of mixed characteristic  with isolated singularity and be  perfectoid-pure. If 
		\(
		\operatorname{Ext}_R^i(k, P) = 0
		\)
		for some perfectoid $R$-algebra $P$ and some \(i > \dim(R)\), then $R$ is Gorenstein.
		In particular,  \(R\) is regular
		provided it is of isolated singularity.
	\end{proposition}
	
	\begin{proof}We follow \cite[5.15]{Elham}.
		Recall that 
		\(\operatorname{Ext}_R^i(R/\mathfrak{m}, R) \subseteq \operatorname{Ext}_R^i(R/\mathfrak{m}, P)\).
		So \(\operatorname{Ext}_R^i(R/\mathfrak{m}, R) = 0\).
		Again, by \cite{Ro2},  \(R\) is Gorenstein. In particular, it is Cohen-Macaulay.    This and \(\Ext_R^i(k,P)=0\) enable us to use \cite[page 226]{sc3}, and deduce that \(P\) has finite injective dimension. By \cite[Theorem 3.15]{Ryo}\footnote{or, alternatively,
			by Gorenstein property, we see \(P\) has finite flat dimension. In particular,  \(\operatorname{Tor}_i^R(R/\mathfrak{m}, P) = 0\) for some $i$. By \cite[4.4]{bhatt2019Regular} \(R\) is regular.}, \(R\) is regular.
	\end{proof}

	\begin{example}Let $p > 0$ be a prime and let $k$ be a perfect field of characteristic $p$.  
		Let 
		\(
		f(x_1, \dots, x_n) \in \mathbb{Z}[x_1, \dots, x_n]
		\)
		be a homogeneous polynomial of degree $d \le n$, such that \emph{not all} coefficients of $f$ are divisible by $p$.
		Consider the mixed-characteristic local ring
		\(
		R := W(k)[\![x_2, \dots, x_n]\!] \big/ \big( f(p, x_2, \dots, x_n) \big),
		\)
		where $W(k)$ denotes the ring of Witt vectors over $k$.  
		If 
		\(
		\operatorname{Ext}_R^i(k, R_{perfd}) = 0
		\)
		for  some \(i \geq n\), then $\Deg(f)=1$, and so  \(R\) is regular.
	\end{example}
	
	\begin{proof}
		According to \cite[Example 6.7]{ket}  $R$ is perfectoid-pure. Now, apply {Proposition} \ref{m515}.
	\end{proof}
	Let us remove the  purity assumption from Proposition \ref{m515}.
	
	\begin{corollary}\label{isom}
		Let $(R, \mathfrak{m})$ be of mixed characteristic with an isolated singularity and let $P\neq \fm P$ be a  perfectoid $R$-algebra. Suppose that for some $i \geq \dim(R)$  we have
		\(
		\operatorname{Ext}_{R}^i(k, P) = 0.
		\)
		Then $R$ is regular.
	\end{corollary}
	\begin{proof}We follow \cite[5.17]{Elham}.
		The result of \cite[1.2]{IYENGAR} applies to \(\operatorname{Ext}_R^i(R/\mathfrak{m}, P) = 0\) yields that \(\operatorname{Ext}_R^j(R/\mathfrak{m}, P) = 0\) for all \(j\ge i\). 
		This   enables us to use \cite[Theorem VI.9]{sc3}, and deduce that \(P\) has finite injective dimension. 
		By \cite[Theorem 3.15]{Ryo} the ring $R$ is regular as claimed.\end{proof}
	
	Similarly, we show:
	
	\begin{corollary}
		Let $(R, \mathfrak{m})$ be of mixed characteristic and let $P\neq \fm P$ be a  locally perfectoid $R$-algebra. Suppose that for some $i \geq \dim(R)$  we have
		\(
		\operatorname{Ext}_{R}^i(k(\fp), P_{\fp}) = 0
		\) for all $\fp$ in singular locus $R$.
		Then $R$ is regular.
	\end{corollary}
	
	We have noticed that the natural map is
	$W((R/p)^\perf)  \to R$
	under the assumption of p-completeness.
	Recall that lifting a ring of $\operatorname{char}(R)= p>0$ to a $p$-adic ring embeddable into the Witt vectors is quite subtle,
	except when $R/pR$ is regular or a perfect algebra, to see a criterion of liftability to Witt vectors see \cite{SZ,Shi15}. Here is another well-known fact that we found from a course note by Kedlaya:
	
	\begin{fact}
		Let \(R\) be a \(p\)-torsion-free, \(p\)-adically complete ring equipped with a \(p\)-derivation$
		\delta:R\to R.$ Put $A=R/pR$ and let $A^{\perf}$ be the perfection of \(A\). Denote by
		$\iota:A\to A^{\perf}$ the canonical map. Then the coperfection of \(R\) in the category of \(p\)-torsion-free, \(p\)-adically complete \(\delta\)-rings is identified,   with
		$W(A^{\perf}).$ Consequently there is a morphism of \(p\)-complete \(\delta\)-rings
		$\theta_\delta:R\to W(A^{\perf}). $ The morphism is an injection if $A$ is reduced.
	\end{fact}
	
	\begin{proof}
		We use the Joyal, or \(y\)-, coordinates on \(W(A^{\perf})\). Thus an element of \(W(A^{\perf})\) is written
		$
		(y_0,y_1,y_2,\ldots),
		$
		with the property that the \(p\)-derivation on \(W(A^{\perf})\) is given by the shift
		$
		\delta(y_0,y_1,y_2,\ldots)
		=
		(y_1,y_2,y_3,\ldots).
		$
		With this convention, the morphism \(\theta_\delta\) is given by
		$$\theta_\delta(x)
		=
		\left(
		\iota(\overline{x}),
		\iota(\overline{\delta(x)}),
		\iota(\overline{\delta^2(x)}),
		\iota(\overline{\delta^3(x)}),
		\ldots
		\right),
		$$ 
		where \(\overline{x}\) denotes the image of \(x\) in \(R/pR\). Equivalently,
		$
		y_n\bigl(\theta_\delta(x)\bigr)
		=
		\iota (\overline{\delta^n(x)})
		.
		$
		The \(p\)-derivation \(\delta\) determines the Frobenius lift	
		$
		\varphi(x)=x^p+p\delta(x).
		$
		The reduction map
		$
		R\to A,
		$
		followed by \(\iota:A\to A^{\perf}\), gives a ring homomorphism
		$
		R\to A^{\perf},
		.
		$
		Since \(A^{\perf}\) is perfect of characteristic \(p\), tilting identifies the corresponding \(p\)-complete perfect \(\delta\)-ring with
		$
		W(A^{\perf}).
		$
		The resulting lift is the morphism
		$
		\theta_\delta:R\to W(A^{\perf}).
		$ It remains to determine its coordinates. Write
		$
		\theta_\delta(x)
		=
		(y_0,y_1,y_2,\ldots)
		$ in Joyal coordinates. Since the \(p\)-derivation on \(W(A^{\perf})\) is the shift,
		$
		y_n
		=
		y_0\circ\delta^n.
		$
		Because \(\theta_\delta\) is a morphism of \(\delta\)-rings,
		$\delta^n(\theta_\delta(x))=
		\theta_\delta(\delta^n(x)).$
		Taking zeroth coordinates gives
		$y_n(\theta_\delta(x))
		=
		y_0(\theta_\delta(\delta^n(x))).
		$ By the definition of the lift, its zeroth coordinate is reduction modulo \(p\), followed by \(\iota\). Hence
		$y_n(\theta_\delta(x))
		=
		\iota(\overline{\delta^n(x)}).$
		This implies the desired claim.
	\end{proof}

	\begin{remark}
		Let $R$ be a Noetherian local domain containing   such that $p \in \mathfrak{m} \setminus \mathfrak{m}^2$. Recall from \cite{ket} the first three items, which will be used in what follows.
		\begin{enumerate}
			\item[(H1)] The pair $(R, p)$ is \textit{perfectoid-pure}, meaning there exists a perfectoid $R$-algebra $B$ containing a fixed compatible system of $p$-power roots of $p$ in $B$, such that the induced map
			\(
			pR \to (pB)_{\mathrm{perfd}}
			\)
			is pure as a map of $R$-modules.
			
			\item[(H2)]The pair $(R, p)$ is \textit{lim-perfectoid-pure}, meaning  The induced map
			\(
			pR \to (p)_{\mathrm{perfd}}
			\)
			is pure in the derived category $D(R)$.
			
			\item[(H3)] The pair $(R, p)$ is \textit{lim-perfectoid injective} if the induced map
			\( 
			H^i_{\mathfrak{m}}(pR) \to H^i_{\mathfrak{m}}((p)_{\mathrm{perfd}})
			\)
			is injective for every $i$ and every maximal ideal $\mathfrak{m}$.
			
			\item[(H4)] The pair $(R, p)$ is \textit{d-pure}, meaning there  that     there is a morphism
			\(
			R \to
			\W((R/(p))^\infty)\)
			is pure as a map of $R$-modules.
		\end{enumerate}    
	\end{remark}	
	
	\begin{proposition}
		\label{mainregularity}
		Under hypotheses (H1)--(H4), if
		\(
		\operatorname{Ext}_R^i(R/\mathfrak{m}, (pB)_{\mathrm{perfd}}) = 0
		\)
		for some $i > d := \dim R$, then $R$ is regular.   
	\end{proposition}

	Instead of $(H_4)$ one may assume for the displayed algebra $B$, we have $B/pB$ is minimal perfect algebra over
	$R/pR$. In fact by \cite[Proposition 3.12]{SZ}  such an abject is unique up to isomorphism.
	\begin{proof}
		Since the map $pR \to (pB)_{\mathrm{perfd}}$ is pure by (H1), we obtain $\operatorname{id}(pR) < \infty$, and because $R \cong pR$ as $R$-modules, it follows that $\operatorname{id}(R) < \infty$; hence $R$ is Gorenstein. Meanwhile, by \cite[Proposition~6.5]{ket}, the reduction $R/(p)$ is lim-perfectoid-injective, and since $\operatorname{char}(R/(p)) = p$, this translates precisely into $F$-injectivity in the positive-characteristic sense. As $R$ is Gorenstein, its quotient $R/pR$ is also Gorenstein; a Gorenstein ring that is $F$-injective is automatically $F$-pure, so $R/pR$ is $F$-pure. Now \cite[Theorem 10.77]{Rot} allows us to descend the Ext-vanishing to the special fiber: we obtain
		\(
		\operatorname{Ext}_{R/p}^{i-1}(R/\mathfrak{m}, (pB)_{\mathrm{perfd}}/p) = 0
		\)
		and note that $i-1 > d-1 = \dim(R/(p))$.  There is a natural isomorphism $$\W((R/(p))^\infty)/p\W((R/(p))^\infty)\simeq (R/(p))^\infty.$$ 
		Since $R/(p)$ is $F$-pure and satisfies this Ext-vanishing, the positive-characteristic criterion (see Proposition~\ref{pur}) implies that $R/(p)$ is regular. Finally, because $p \notin \mathfrak{m}^2$, the regularity of the special fiber lifts to the whole ring.  Therefore, $R$ is regular.
	\end{proof}
	
	We believe the validity of the above result holds true without any use of the $(H_i)$ conditions.

	\section{Idealwise separatedness of \texorpdfstring{$R^+$}{R+}}
	
	The main result of this section is 
	{Theorem} \ref{thm:regular-cover-idealwise-separated}.
	We begin by recalling the following facts.
	
	\begin{fact}\label{thm:regular-plus-idealwise-separated}
		Let $(R,\fm,k)$ be an excellent regular local domain of characteristic
		$p>0$. Then $R^+$ is flat and $\fm$-adically idealwise separated as an
		$R$-module.
	\end{fact}
	
	\begin{proof}
		By the theorem of Hochster and Huneke, $R^+$ is a balanced big
		Cohen--Macaulay $R$-algebra; see \cite{HH}. Since $R$ is regular,
		every balanced big Cohen--Macaulay $R$-module is flat; see
		\cite{Hoc83}.
		Consequently, $R^+$ is flat over $R$.
		On the other hand, $R^+$ is $\fm$-adically separated. Therefore
		Proposition~\ref{prop:flat-separated-idealwise} applies and shows that
		\(
		I\otimes_RR^+
		\)
		is $\fm$-adically separated for every ideal $I\subseteq R$. Hence
		$R^+$ is $\fm$-adically idealwise separated.
	\end{proof}
	
	\begin{remark}\label{rem:hh-does-not-give-total}
		Fact~\ref{thm:regular-plus-idealwise-separated} proves idealwise
		separatedness, but it does not by itself prove that $R^+$ is totally
		separated, Ohm--Rush, or intersection flat. Those properties are
		usually strictly stronger. In particular, flatness alone does not
		imply any of them.
	\end{remark}
	
	We note the following theorem for the convenience of the reader.
	
	\begin{theorem}\label{thm:pure-and-1dcase}(See \cite{Shravan}).
		Algebras over $(R,\mathfrak m)$ which are a ``limit of split maps" are indeed $\fm$-adically idealwise separated. For example, \begin{itemize}
			\item[(a)] Let $R$ be a $1$-dimensional excellent domain. Then $R^+$ is $\fm$-adically idealwise separated
			\item[(b)] Let $R$ be   $F$-pure. Then $R^{\infty}$ is $\fm$-adically idealwise separated.\end{itemize}
	\end{theorem}
	
	\begin{proof}
		In case {\rm (a)}, write $R^+$ as the directed union of the module-finite
		normal extensions of $R$. The transition maps split: the corresponding quotients are finite
		torsion-free, hence projective since they are (semi-local) Dedekind domains.
		In case {\rm (b)}, recall that
		\(
		R^\infty=\varinjlim_e R^{1/p^e}.
		\)
		Since $R$ is $F$-pure, the transition maps are pure.
		Thus any element of
		\(
		\bigcap_{n\geq 1}\mathfrak m^n IR^\infty
		\)
		may be brought to a finite $R$-module containing it. There
		it belongs to the intersection of all its $\mathfrak m$-adic powers, which
		is zero by the Krull intersection theorem. Hence $R^\infty$ is
		$\mathfrak m$-adically idealwise separated.
	\end{proof}

	\begin{remark}\label{rem:bhargavisawesome}
		Note that the slogan of Theorem \ref{thm:pure-and-1dcase} is \textbf{not} satisfied for $R^{+}$. Bhatt gives examples of non-existence of small Cohen-Macaulay algebras \cite{BhattSmallCM}, and since splinters are Cohen-Macaulay (in positive characteristic) $R^{+}$ cannot be a limit of split maps.
	\end{remark}
	
	We hope to pursue this question in the future:
	
	\begin{question}
		In light of Remark \ref{rem:bhargavisawesome}, can one use \emph{lim Cohen-Macaulay sequences} (see \cite{BHM}) to study  $(Q_3)$ from the introduction?  
	\end{question}
	
	As outlined in previous items, it is well known that $R^{+}$ is $\fm$-adically separated and the question of whether it is $\fm$-adically \emph{idealwise} separated is intimately tied to the homological questions we discussed. This question has deep implications - a positive answer \textit{implies} \cite[Theorem 4.13]{bhatt2019Regular}, which \emph{a priori} uses Cohen-Macaulayness of $R^{+}$. We record a result outside the regular case.
	
	\begin{theorem}\label{thm:regular-cover-idealwise-separated}
		Let $(R,\fm,k)$ be an excellent local domain of characteristic $p>0$. Suppose
		that there is a module-finite local extension
		\(
		(R,\fm)\to(S,\fn)
		\)
		such that $S$ is regular and $S/R$ is flat on the punctured spectrum of $R$.
		Then $R^+$ is $\fm$-adically idealwise separated as an $R$-module.
	\end{theorem}
	
	\begin{proof}
		
		Summary: using Fact \ref{thm:regular-plus-idealwise-separated} (which uses the Cohen-Macaulayness of $R^{+}$) we know that $R^{+} = S^{+}$ is $\fm$-adically idealwise separated as an $S$-module. However the failure of $R^{+}$ being $m$-adically idealwise separated as an $R$ module is captured by the kernel of \(
		\mu_I\colon I\otimes_RS\to IS
		\) 
		whose kernel is of finite length. Since we are taking infinite intersections of powers of $\fm$ an Artin-Rees argument does the job.  
		
		Fix an algebraic closure $L$ of $\Frac(S)$. Since $\Frac(S)$ is algebraic
		over $\Frac(R)$, the field $L$ is also an algebraic closure of
		$\Frac(R)$. Thus, upon taking both absolute integral closures inside $L$,
		we have
		\(
		S^+=R^+.
		\)
		Since $S$ is module-finite over the excellent ring $R$, it is excellent. The
		theorem of Hochster and Huneke implies that $S^+$ is a balanced big
		Cohen--Macaulay $S$-algebra; see \cite{HH}. Since $S$ is regular, every balanced big Cohen--Macaulay $S$-module is
		flat over $S$; see \cite{Hoc83}. Consequently,
		$S^+=R^+$ is flat over $S$.
		Let $I\subseteq R$ be an ideal. We must prove that
		\(
		I\otimes_RR^+
		\)
		is $\fm$-adically separated. Consider the natural multiplication map of finite
		$S$-modules
		\(
		\mu_I\colon I\otimes_RS\to IS
		\)
		and set
		\(
		A_I:=I\otimes_RS\)
		{and} \(K_I:=\ker(\mu_I).\)
		We first show that $K_I$ has finite length. Let $\fp\neq\fm$. By assumption, $(S/R)_{\fp}$ is flat over $R_{\fp}$.
		Localizing the exact sequence
		\(
		0\to R\to S\to S/R\to 0
		\)
		at $\fp$, we see that $S_{\fp}$ is an extension of two flat
		$R_{\fp}$-modules, and hence is flat over $R_{\fp}$. Therefore the
		multiplication map
		\(
		I_{\fp}\otimes_{R_{\fp}}S_{\fp}
		\to
		I_{\fp}S_{\fp}
		\)
		is injective. Consequently,
		\(
		(K_I)_{\fp}=0.
		\)
		Thus
		\(
		\operatorname{Supp}_R(K_I)\subseteq\{\fm\}.
		\)
		Since $K_I$ is finitely generated over $R$, it has finite length. In
		particular, there is an integer $e>0$ such that
		\(
		\fm^eK_I=0.
		\)
		Tensor the exact sequence
		\[
		0\longrightarrow K_I\longrightarrow A_I
		\xrightarrow{\mu_I}IS\longrightarrow0
		\]
		over $S$ with the flat $S$-algebra $S^+$. We obtain an exact sequence
		\[
		0\longrightarrow K_I\otimes_SS^+
		\longrightarrow A_I\otimes_SS^+
		\longrightarrow IS\otimes_SS^+
		\longrightarrow0 \quad(\ast).
		\]
		By associativity of tensor products and the equality $S^+=R^+$, the middle
		term is naturally identified with
		\(
		A_I\otimes_SS^+
		\cong
		I\otimes_RS^+
		=
		I\otimes_RR^+.
		\)
		Moreover, since $S^+$ is flat over $S$
		\(
		IS\otimes_SS^+
		\cong
		IS^+
		\subseteq
		S^+.
		\)
		By {Fact}~\ref{thm:regular-plus-idealwise-separated}, $S^+$ is $\fn$-adically separated. Since
		$\fm S\subseteq\fn$, we have
		\(
		\fm^nS^+\subseteq\fn^nS^+
		\)
		for every $n$. It follows that $S^+$, and hence its submodule $IS^+$, is
		$\fm$-adically separated.
		It remains to control the kernel $K_I\otimes_SS^+$. 
		Apply the Artin-Rees
		lemma to the inclusion
		\(
		K_I\subseteq A_I
		\)
		of finite $S$-modules and the ideal $\fm S$. There is an integer $c>0$ such
		that, for every $n\geq c$,
		\(
		K_I\cap\fm^nA_I
		\subseteq
		\fm^{n-c}K_I.
		\)
		Since $\fm^eK_I=0$, it follows that
		\(
		K_I\cap\fm^nA_I=0\)
		{for all } \(n\geq c+e.
		\)
		Flat base change preserves intersections of two submodules. Hence, inside
		$A_I\otimes_SS^+$, we have
		\[
		\bigl(K_I\otimes_SS^+\bigr)
		\cap
		\fm^n\bigl(A_I\otimes_SS^+\bigr)
		=
		\bigl(K_I\cap\fm^nA_I\bigr)\otimes_SS^+.
		\]
		Consequently,
		\(
		\bigl(K_I\otimes_SS^+\bigr)
		\cap
		\fm^n\bigl(A_I\otimes_SS^+\bigr)
		=0\)
		{for all }\(n\gg0.
		\)
		Now let
		\(
		\alpha\in
		\bigcap_{n\geq1}
		\fm^n\bigl(A_I\otimes_SS^+\bigr).
		\)
		Its image in $IS\otimes_SS^+$ belongs to
		\(
		\bigcap_{n\geq1}
		\fm^n(IS\otimes_SS^+)=0,
		\)
		so
		\(
		\alpha\in K_I\otimes_SS^+ ,
		\) see $(\ast)$.
		On the other hand, $\alpha$ belongs to
		\(
		\fm^n\bigl(A_I\otimes_SS^+\bigr)
		\)
		for every $n$. Taking $n\gg0$ in the preceding intersection equality gives
		$\alpha=0$. Therefore
		\(
		\bigcap_{n\geq1}\fm^n(I\otimes_RR^+)=0.
		\)
		Since $I\subseteq R$ was arbitrary, $R^+$ is $\fm$-adically idealwise
		separated.
	\end{proof}
	
	\begin{example}\label{ver}
		The Veronese subring  $R$ of a polynomial ring $S$ over a field is regular at every prime ideal except at the unique homogeneous maximal ideal. Recall from \cite[23.1]{mat}
		that $S/R$ is flat on the punctured spectrum of $R$.
		Thus,  Veronese rings are sample of Theorem~\ref{thm:regular-cover-idealwise-separated}, and so $R^+$  is
		$\fm$-adically idealwise separated.
	\end{example}

	\begin{corollary}\label{cor:codimension-one-regular-cover}
		Let $(R,\fm)$ be an excellent local domain of characteristic $p>0$ with
		$\dim(R)\leq2$. Suppose that there is a module-finite local extension
		\(
		R\to(S,\fn)
		\)
		such that $S$ is regular and $S/R$ is flat in codimension one. Then $R^+$ is
		$\fm$-adically idealwise separated.
	\end{corollary}
	
	\begin{proof}
		Every prime $\fp\neq\fm$ has height at most one. Hence $S/R$ is flat on the
		punctured spectrum, and the assertion follows from
		Theorem~\ref{thm:regular-cover-idealwise-separated}.
	\end{proof}
	
	\begin{corollary}
		\label{cor:idealwise-separated}
		Let $R$ be an excellent $\ast$-local domain of prime characteristic with isolated singularity such that its integral closure is regular. Then  $R^+$ is $\mathfrak{m}$-adically idealwise separated.
	\end{corollary}
	
	\begin{proof} 
		Let $S:=\overline{R}$.
		The desired property follows  from Theorem~\ref{thm:regular-cover-idealwise-separated}, as $S/R$ is flat on the punctured spectrum of $R$.
	\end{proof}
	
	\begin{remark}
		Of course, there are interesting classes of examples of singularities that satisfy the hypothesis of Theorem \ref{thm:regular-cover-idealwise-separated} and  Corollary \ref{cor:codimension-one-regular-cover}, Veronese rings (see {Example} \ref{ver}), isolated quotient singularities, all with suitable tameness assumptions ($p$ does not divide the cardinality of the group or extension), and rational double points, without any tameness assumptions, see below.
	\end{remark}
	
	\begin{corollary}\label{cor:rdp-idealwise-separated}
		Let $(R,\fm)$ be a complete rational double point over an algebraically
		closed field of characteristic $p>0$. Then $R^+$ is $\fm$-adically
		idealwise separated.
	\end{corollary}
	
	\begin{proof}
		By Artin \cite[p.~11]{ArtRDP}, there is a module-finite
		local extension $R\to S$ with $S$ regular. Since $R$ is a two-dimensional
		normal domain, $R_{\fp}$ is a field or a DVR for every $\fp\neq\fm$.
		Moreover, $(S/R)_{\fp}$ is torsion-free: if $0\neq a\in R_{\fp}$ and
		$as\in R_{\fp}$, then $s\in\Frac(R_{\fp})$ is integral over $R_{\fp}$,
		hence $s\in R_{\fp}$. Thus $S/R$ is flat on the punctured spectrum, and
		Theorem~\ref{thm:regular-cover-idealwise-separated} applies.
	\end{proof}
	
	The above corollaries suggest the following.
	
	\begin{question}
		If $R\to S$ is a finite extension of local domains, flat outside the closed point. If $S^{+}$ ($= R^{+}$) is $\mathfrak{m}$-adically idealwise separated as a $S$-module, is $R^{+}$ $\mathfrak{m}$-adically idealwise separated as a $R$-module?
	\end{question}
	
	\begin{remark}
		Sanity check - Note that if $R \to S$ is flat in the above setup, then $K_I = 0$ and it follows that $I \otimes_{R} R^{+} \cong I \otimes_{R} S \otimes_{S} S^{+} \cong IS \otimes_{S} S^{+}$ which is separated. 
	\end{remark}

\end{document}